\documentclass[12pt]{article}
\usepackage{amssymb}
\usepackage{amsmath,amsthm}
\usepackage[latin1]{inputenc}
\usepackage[dvips]{graphicx}
\usepackage{hyperref}
\usepackage{enumerate}

\hypersetup{colorlinks=true}

\hypersetup{colorlinks=true, linkcolor=blue, citecolor=blue,urlcolor=blue}

 \newtheorem{remark}{Remark}

 \newtheorem{lemma}[remark]{Lemma}
 \newtheorem{theorem}[remark]{Theorem}
 \newtheorem{proposition}[remark]{Proposition}
 \newtheorem{corollary}[remark]{Corollary}

\title{On the strong metric dimension of Cartesian and direct products of graphs}

\author{Juan A. Rodr\'{\i}guez-Vel\'{a}zquez$^{(1)}$, Ismael G. Yero$^{(2)}$,\\
Dorota Kuziak$^{(1)}$ and Ortrud R. Oellermann$^{(3)}$\\
$^{(1)}${\small Departament d'Enginyeria Inform\`atica i Matem\`atiques,}\\
{\small Universitat Rovira i Virgili,}  {\small Av. Pa\"{\i}sos
Catalans 26, 43007 Tarragona, Spain.} \\{\small
juanalberto.rodriguez\@@urv.cat, dorota.kuziak\@@urv.cat}
\\
$^{(2)}${\small Departamento de Matem\'aticas, Escuela Polit\'ecnica Superior de Algeciras}\\
{\small Universidad de C\'adiz,} {\small
Av. Ram\'on Puyol s/n, 11202 Algeciras, Spain.} \\ {\small
ismael.gonzalez\@@uca.es}\\
$^{(3)}${\small Department of Mathematics and Statistics, University of Winnipeg}\\{\small Winnipeg, MB R3B 2E9, Canada.}\\{\small o.oellermann@uwinnipeg.ca}\\}

\begin{document}
\maketitle

\begin{abstract}
Let $G$ be a connected graph. A vertex $w$ {\em strongly resolves} a pair $u, v$ of vertices of $G$ if there exists some shortest
$u-w$ path containing $v$ or some shortest $v-w$ path containing $u$. A set $W$ of vertices is a {\em strong resolving set} for $G$ if every pair of
vertices of $G$ is strongly resolved by some vertex of $W$. The smallest cardinality of a strong resolving set for $G$ is called the {\em strong metric dimension} of $G$.
It is known  that the problem of computing the strong metric
dimension of a graph is NP-hard. In this paper we obtain closed formulae for the strong metric dimension of several families of Cartesian product graphs and direct product graphs.
\end{abstract}

{\it Keywords:} Strong resolving set; strong metric dimension; Cartesian product graph; direct product graph, strong resolving graph.

{\it AMS Subject Classification Numbers:}   05C12; 05C76; 05C90.

\section{Introduction}

A {\em generator} of a metric space is a set $S$ of points in the space with the property that every point of the space is uniquely determined by its distances from the elements of $S$. Given a simple and connected graph $G=(V,E)$, we consider the metric $d_G:V\times V\rightarrow \mathbb{R}^+$, where $d_G(x,y)$ is the length of a shortest path between $u$ and $v$. $(V,d_G)$ is clearly a metric space. A vertex $v\in V$ is said to distinguish two vertices $x$ and $y$ if $d_G(v,x)\ne d_G(v,y)$.
A set $S\subset V$ is said to be a \emph{metric generator} for $G$ if any pair of vertices of $G$ is
distinguished by some element of $S$. A minimum generator is called a \emph{metric basis}, and
its cardinality the \emph{metric dimension} of $G$, denoted by $dim(G)$. Motivated by the problem of uniquely determining the location of an intruder in a network, the concept of metric
dimension of a graph was introduced by Slater in \cite{leaves-trees,slater2}, where the metric generators were called \emph{locating sets}. The concept of metric dimension of a graph was introduced independently by Harary and Melter in \cite{harary}, where metric generators were called \emph{resolving sets}. Applications of this invariant to the navigation of robots in networks are discussed in \cite{landmarks} and applications to chemistry in \cite{pharmacy1,pharmacy2}.  This invariant was studied further in a number of other papers including for example, \cite{pelayo1, chappell,chartrand,chartrand2, fehr,haynes,Tomescu1,LocalMetric,survey,tomescu,yerocartpartres,CMWA}.  Several variations of metric generators including resolving dominating sets \cite{brigham}, independent resolving sets \cite{chartrand1}, local metric sets \cite{LocalMetric}, and strong resolving sets \cite{seb}, etc. have since been introduced and studied.

In this article we are interested in the study  of strong resolving sets \cite{Oellermann,seb}.
For two vertices $u$ and $v$ in a connected graph $G$, the interval $I_G [u, v]$ between $u$ and $v$ is defined as the collection of all vertices that belong to some shortest $u-v$ path. A vertex $w$ strongly resolves two vertices $u$ and $v$ if $v\in  I_G [u,w]$ or  $u\in I_G [v,w]$. A set $S$ of vertices in a connected graph $G$ is a \emph{strong resolving set} for $G$ if every two vertices of $G$ are strongly resolved by some vertex
of $S$. The smallest cardinality of a strong resolving set of $G$ is called \emph{strong metric dimension} and is denoted by $dim_s(G)$.  So, for example, $dim_s(G)=n-1$ if and only if $G$ is the complete graph  of order $n$.
For the cycle $C_n$ of order $n$ the strong dimension is $dim_s(C_n) = \lceil n/2\rceil$  and if $T$ is a tree with $l(T)$ leaves, its strong metric dimension equals  $l(T)-1$   (see \cite{seb}).  We say that a strong resolving set for $G$ of cardinality $dim_s(G)$ is a \emph{strong metric basis} of $G$.

A vertex $u$ of $G$ is \emph{maximally distant} from $v$ if for every vertex $w$ in the open neighborhood of $u$, $d_G(v,w)\le d_G(u,v)$. If $u$ is maximally distant from $v$ and $v$ is maximally distant from $u$, then we say that $u$ and $v$ are \emph{mutually maximally distant}.  The {\em boundary} of $G=(V,E)$ is defined as $\partial(G) = \{u\in V:$ there exists $v\in V$  such that $u,v$  are mutually maximally distant$\}$. For some basic graph classes, such as complete graphs $K_n$, complete bipartite graphs $K_{r,s}$,  cycles $C_n$ and hypercube graphs $Q_k$, the boundary is simply the whole vertex set.
It is not difficult to see that this property holds for all  $2$-antipodal\footnote{The diameter of $G=(V,E)$ is defined as $D(G)=\max_{u,v\in V}\{d(u,v)\}$.  We recall that $G=(V,E)$ is $2$-antipodal if for each vertex $x\in V$ there exists exactly one vertex $y\in V$ such that $d_G(x,y)=D(G)$.} graphs and also for all distance-regular graphs.
Notice that the boundary of a tree consists exactly of the set of its leaves. A vertex  of a graph  is a {\em simplicial vertex} if the subgraph induced by  its neighbors is a complete graph. Given a graph $G$, we denote by $\varepsilon(G)$ the set of simplicial vertices of $G$.
Notice that $\sigma(G)\subseteq \partial(G)$.

We use the notion of strong resolving graph introduced in \cite{Oellermann}. The \emph{strong resolving graph}\footnote{In fact, according to \cite{Oellermann} the strong resolving graph $G'_{SR}$ of a graph $G$ has vertex set $V(G'_{SR})=V(G)$ and two vertices $u,v$ are adjacent in $G'_{SR}$ if and only if $u$ and $v$ are mutually maximally distant in $G$. So, the strong resolving graph defined here is a subgraph of the strong resolving graph defined in \cite{Oellermann} and can be obtained from the latter graph by deleting its isolated vertices.} of $G$ is a graph $G_{SR}$  with vertex set $V(G_{SR}) = \partial(G)$ where two vertices $u,v$ are adjacent in $G_{SR}$ if and only if $u$ and $v$ are mutually maximally distant in $G$.

There are some families of graph for which its resolving graph can be obtained relatively easily. For instance, we emphasize the following cases.
\begin{itemize}
\item If $\partial(G)=\sigma(G)$, then $G_{SR}\cong K_{\partial(G)}$. In particular, $(K_n)_{SR}\cong K_n$ and for any tree $T$ with $l(T)$ leaves, $(T)_{SR}\cong K_{l(T)}$.

\item For any $2$-antipodal graph $G$ of order $n$, $G_{SR}\cong \bigcup_{i=1}^{\frac{n}{2}} K_2$. In particular,  $(C_{2k})_{SR}\cong \bigcup_{i=1}^{k} K_2$.

\item $(C_{2k+1})_{SR}\cong C_{2k+1}$.
\end{itemize}

A set $S$ of vertices of $G$ is a \emph{vertex cover} of $G$ if every edge of $G$ is incident with at least one vertex of $S$. The \emph{vertex cover number} of $G$, denoted by $\alpha(G)$, is the smallest cardinality of a vertex cover of $G$. We refer to an $\alpha(G)$-set in a graph $G$ as a vertex cover  of cardinality $\alpha(G)$. Oellermann and Peters-Fransen \cite{Oellermann} showed that the problem of finding the strong metric dimension of a connected graph $G$ can be transformed to the problem of finding the vertex cover number of $G_{SR}$.

\begin{theorem}{\em \cite{Oellermann}}\label{lem_oellerman}
For any connected graph $G$,
$dim_s(G) = \alpha(G_{SR}).$
\end{theorem}

It was shown in \cite{Oellermann}  that the problem of computing $dim_s(G)$ is NP-hard. This suggests finding the strong metric dimension for special classes of graphs or obtaining good bounds on this invariant. An efficient procedure for finding the strong metric dimension of distance hereditary graphs was described in \cite{may-oellermann}. In this paper we study the problem of finding exact values or sharp bounds for the strong metric dimension of Cartesian and direct products of graphs and express these in terms of invariants of the factor graphs.

\section{The strong metric dimension of Cartesian products of graphs}
We recall that the \emph{Cartesian product of two graphs} $G=(V_1,E_1)$ and
$H=(V_2,E_2)$ is the graph $G\Box H$, such that
$V(G\Box H)=V_1\times V_2$ and two vertices $(a,b),(c,d)$ are adjacent in  $G\Box H$ if and only if, either ($a=c$ and $bd\in E_2$)
or ($b=d$ and $ac\in E_1$).

The following result will be useful to study the  relationship between  the strong resolving graph of $G\Box H$ and the strong resolving graphs of $G$ and $H$.

\begin{lemma}{\em \cite{bres-klav-tepeh}}\label{lem_bound}
For any graphs $G$ and $H$,
$\partial(G\square H)=\partial(G)\times \partial(H).$
\end{lemma}

The \emph{direct product of two graphs} $G=(V_1,E_1)$ and
$H=(V_2,E_2)$ is the graph $G\times H$, such that
$V(G\times H)=V_1\times V_2$ and two vertices $(a,b),(c,d)$
 are adjacent in  $G\times H$ if and only if $ac\in E_1$ and $bd\in E_2$.

\begin{theorem}\label{cartesian_directed}
Let $G$ and $H$ be two connected graphs. Then $(G\square H)_{SR}\cong G_{SR}\times H_{SR}$.
\end{theorem}

\begin{proof}
By definition of strong resolving graph and Lemma \ref{lem_bound} we have
 $$V((G\square H)_{SR})=\partial (G\square H)=\partial(G)\times \partial(H)=V(G_{SR}\times H_{SR}).$$
 Moreover, given $x,y\in \partial(G)$ and $a,b\in \partial(H)$ the following assertions are equivalent:
 \begin{enumerate}[(i)]
  \item  $(x,a), (y,b)$ are adjacent in $G_{SR}\times H_{SR}$.
  \item $x,y$ are adjacent in $G_{SR}$ and $a,b$ are adjacent in $H_{SR}$.
  \item $x,y$ are mutually maximally distant in $G$ and $a,b$ are mutually maximally distant in $H$.
  \item  $(x,a),(y,b)$ are mutually maximally distant in $G\square H$.
  \item  $(x,a),(y,b)$ are adjacent in $(G\square H)_{SR}$.
 \end{enumerate}
 We only need to show that (iii) $\Leftrightarrow$ (iv). Let  $x,y$ be two mutually maximally distant vertices in $G$ and let $a,b$ be two mutually maximally distant vertices in $H$. If $(z,b)$ is adjacent to $(y,b)$ in $G\square H$, then $z$ is adjacent to $y$ in $G$. Thus,
 $$d_{G\square H}((x,a),(z,b))=d_G(x,z)+d_H(a,b)\le d_G(x,y)+d_H(a,b)=d_{G\square H}((x,a),(y,b)).$$
 Analogously we see that if $(y,c)$ adjacent to $(y,b)$ in $G\square H$, then $d_{G\square H}((x,a),(y,c))\le d_{G\square H}((x,a),(y,b)).$ Thus, $(y,b)$ is maximally distant from $(x,a)$.
 Analogously it can be shown that $(x,a)$ is maximally distant from $(y,b)$. Hence, (iii)$\Rightarrow$(iv).

Now, let $(x,a),(y,b)$ be two  mutually maximally distant vertices in $G\square H$. Suppose that there exists $z$ belonging to the open neighborhood of $y$ such that $d_G(x,y)<d_G(x,z)$. In this case $(z,b)$ belongs to the open neighborhood of $(y,b)$ and  $$d_{G\square H}((x,a),(y,b))=d_G(x,y)+d_H(a,b)< d_G(x,z)+d_H(a,b)=d_{G\square H}((x,a),(z,b)),$$
 which is a contradiction. Hence, $y$ is maximally distant from $x$ in $G$. Analogously it can be shown that  $x$ is maximally distant from $y$. By symmetry the vertices $a,b$ are mutually maximally distant in $H$. Therefore, (iv)$\Rightarrow$(iii).
\end{proof}

The following result, which is a direct consequence of Theorem \ref{lem_oellerman} and Theorem \ref{cartesian_directed}, will be the main tool of this section.
\begin{theorem}\label{eq_oellerman_direct}
Let $G$ and $H$ be two connected graphs. Then $dim_s(G\square H) = \alpha (G_{SR}\times H_{SR})$.
\end{theorem}

Now we consider some cases in which we can compute  $\alpha (G_{SR}\times H_{SR})$. To begin with we recall the following well-known result.

\begin{theorem}{\rm (K\"{o}nig, Egerv\'{a}ry, 1931)} For bipartite graphs the size of a maximum matching
equals the size of a minimum vertex cover.
\end{theorem}\label{konig-egervary}

We will use K\"{o}nig-Egerv\'{a}ry's theorem to prove the following result.

\begin{theorem}\label{theoremMatching}
Let $G$ and $H$ be two connected graphs such that $H_{SR}$ is bipartite with a perfect matching.
Let $G_i$, $i\in \{1,...,k\}$, be the connected components of $G_{SR}$. If for each $i\in \{1,...,k\}$, $G_i$ is Hamiltonian or $G_i$ has a perfect matching, then
$$dim_s(G\square H)=\frac{|\partial(G)||\partial(H)|}{2}.$$
\end{theorem}

\begin{proof}
Since $H_{SR}$ is bipartite, $G_{SR}\times H_{SR}$ is bipartite. We show next that $G_{SR}\times H_{SR}$ has a perfect matching.

Let  $n_i$ be the order of $G_i$, $i\in \{1,...,k\}$, and let $\{x_1y_1,x_2y_2, ...,x_{\frac{|\partial(H)|}{2}}y_{\frac{|\partial(H)|}{2}}\}$ be a perfect matching of  $H_{SR}$.
We differentiate two cases.

Case 1: $G_i$ has a perfect matching. If  $\{u_1v_1,u_2v_2 ...,u_{\frac{n_i}{2}}v_{\frac{n_i}{2}}\}$ is a perfect matching of $G_i$, then the set
$\{(u_1,y_1)(v_1,x_1),(v_1,y_1)(u_1,x_1)$,...,$(u_{\frac{n_i}{2}},y_1)$ $(v_{\frac{n_i}{2}},x_1)$, $(v_{\frac{n_i}{2}},y_1)(u_{\frac{n_i}{2}},x_1)$  ,..., $(u_1,y_2)(v_1,x_2),$ $(v_1,y_2)(u_1,x_2)$ ,..., $(u_{\frac{n_i}{2}},y_2)(v_{\frac{n_i}{2}},x_2),$ $(v_{\frac{n_i}{2}},y_2)(u_{\frac{n_i}{2}},x_2),$ ..., $(u_1,y_{\frac{|\partial(H)|}{2}})(v_1,x_{\frac{|\partial(H)|}{2}}),$ $(v_1,y_{\frac{|\partial(H)|}{2}})(u_1,x_{\frac{|\partial(H)|}{2}})$, ...,$(u_{\frac{n_i}{2}},y_{\frac{|\partial(H)|}{2}})(v_{\frac{n_i}{2}},x_{\frac{|\partial(H)|}{2}})$, $(v_{\frac{n_i}{2}},y_{\frac{|\partial(H)|}{2}})(u_{\frac{n_i}{2}},x_{\frac{|\partial(H)|}{2}})\}$
is a perfect matching of $G_i\times H_{SR}$.

Case 2: $G_i$ is Hamiltonian. Let $v_1,v_2,...,v_{n_i},v_1$ be  a Hamiltonian cycle of $G_i$. If $n_i$ is even, then $G_i$ has a perfect matching and this case coincides with Case 1. So we suppose that $n_i$ is odd.  In this case,   $\{(v_1,x_1)(v_2,y_1),(v_2,x_1)(v_3,y_1),$..., $(v_{n_i-1},x_1)(v_{n_i},y_1),$  $(v_{n_i},x_1)(v_1,y_1)$, $(v_1,x_2)(v_2,y_2),(v_2,x_2)(v_3,y_2),$ ..., $(v_{n_i-1},x_2)(v_{n_i},y_2),$ $(v_{n_i},x_2)(v_1,y_2),$ ..., $(v_1,x_{\frac{|\partial(H)|}{2}})(v_2,y_{\frac{|\partial(H)|}{2}}),(v_2,x_{\frac{|\partial(H)|}{2}})(v_3,y_{\frac{|\partial(H)|}{2}}),$ ...,\\ ...,$(v_{n_i-1},x_{\frac{|\partial(H)|}{2}})(v_{n_i},y_{\frac{|\partial(H)|}{2}}),$$ (v_{n_i},x_{\frac{|\partial(H)|}{2}})(v_1,y_{\frac{|\partial(H)|}{2}})\}$
is a perfect matching of $G_i\times H_{SR}$.

According to Cases 1 and 2 the graph $\bigcup_{i=1}^kG_i\times H_{SR}=G_{SR}\times H_{SR}$ has a perfect matching. Now, since $G_{SR}\times H_{SR}$ is bipartite and it has a perfect matching, by the K\"{o}nig-Egerv\'{a}ry Theorem we have $\alpha(G_{SR}\times H_{SR})=\frac{|\partial(G)||\partial(H)|}{2}$. This, concludes the proof by Theorem \ref{eq_oellerman_direct}.
\end{proof}

\begin{corollary}
The following statements hold for any connected $2$-antipodal graph $G$ of order $n$.
\begin{itemize}
\item If $H$ is a connected $2$-antipodal graph of order $r$, then $dim_s(G\square H)=\frac{nr}{2}.$
\item If $H$ is a connected graph where $|\partial(H)|=|\sigma(H)|$, then $$dim_s(G\square H)=\frac{n|\sigma(H)|}{2}.$$ In particular, for any tree $T$, $dim_s(G\square T)=\frac{nl(T)}{2}$, and for any complete graph $K_r$, $dim_s(G\square K_r)=\frac{nr}{2}.$
\end{itemize}
\end{corollary}

Since $(K_n)_{SR}\cong K_n$, $(T)_{SR}\cong K_{l(T)}$, where $T$ is a tree with $l(T)$ leaves, $(C_{2k})_{SR}\cong \bigcup_{i=1}^k K_2$ and $(C_{2k+1})_{SR}\cong C_{2k+1}$, the following is a consequence of Theorem \ref{theoremMatching}:

\begin{corollary}\label{corollaryTheoremMatching}

\begin{enumerate}[{\rm (a)}]
\item  $dim_s(K_n\square P_r)=n.$
\item  For any tree $T$ of $l(T)$ leaves, $dim_s(T\square P_r)=l(T).$
\item  $dim_s(C_{n}\square P_r)=n.$
\item  $dim_s(K_{n}\square C_{2k})=nk.$
\item  For any tree $T$ of $l(T)$ leaves, $dim_s(T\square C_{2k})=l(T)k.$
\item  \cite{Oellermann} $dim_s(C_n\square C_{2k})=nk.$
\end{enumerate}
\end{corollary}

Our next tool is a well-known consequence of  Hall's marriage theorem.
\begin{lemma}\label{lemmaHall's}
{\rm (Hall, 1935)}
Every regular bipartite graph has a perfect matching.
\end{lemma}

\begin{theorem}\label{theoremMatching2}
Let   $G$ and $H$ be two connected graphs such that $G_{SR}$ and $H_{SR}$ are  regular and at least one of them is bipartite. Then
$$dim_s(G\square H)=\frac{|\partial(G)||\partial(H)|}{2}.$$
\end{theorem}

\begin{proof}
Since $G_{SR}$ and $H_{SR}$ are  regular graphs and at least one of them is bipartite,  $G_{SR}\times H_{SR}$ is a regular bipartite graph. Hence, by Lemma \ref{lemmaHall's}, $G_{SR}\times H_{SR}$ has a perfect matching. Thus, by the K\"{o}nig-Egerv\'{a}ry Theorem,  $\alpha(G_{SR}\times H_{SR})=\frac{|\partial(G)||\partial(H)|}{2}$. By Theorem \ref{eq_oellerman_direct} the result follows.
\end{proof}

Note that Corollary \ref{corollaryTheoremMatching} is also  a direct consequence of Theorem \ref{theoremMatching2}.

The following consequence of  Theorem \ref{theoremMatching2} is derived from the fact that the strong resolving graph of a distance-regular graph is regular.

\begin{corollary}\label{corollaryTheoremMatching2}
Let $G$ be a distance-regular graph of order $n$ and let $H$ be a connected graph such that $H_{SR}$ is a regular bipartite graph. Then $$dim_s(G\square H)=\frac{n|\partial(H)|}{2}.$$ In particular, is $H$ is a $2$-antipodal graph of order $r$, then $$dim_s(G\square H)=\frac{nr}{2}.$$
\end{corollary}

Recall that the largest cardinality of a set of vertices of $G$, no two of which are adjacent, is called the \emph{independence number} of $G$ and it is denoted by $\beta(G)$. We refer to a $\beta(G)$-set in a graph $G$ as an independent set of cardinality $\beta(G)$.

The following well-known result, due to Gallai, states the relationship between the independence number and the vertex cover number of a graph.

\begin{theorem}{\rm (Gallai's theorem)}\label{th_gallai}
For any graph  $G$ of order $n$,
$\alpha(G)+\beta(G) = n.$
\end{theorem}
Thus using this result and  Theorem \ref{eq_oellerman_direct} we obtain
\begin{equation}\label{eq_oellerman_gallai}
dim_s(G\square H) = |\partial(G)||\partial(H)| - \beta (G_{SR}\times H_{SR})
\end{equation}

\begin{lemma}\label{lemma-vertex-transitive}{\em \cite{vertex-transitive}}
Let $G$ and $H$ be two  vertex-transitive graphs of order $n_1$, $n_2$, respectively. Then $$\beta(G\times H)=\max\{n_1\beta(H),n_2\beta(G)\}.$$
\end{lemma}

\begin{theorem}\label{vertex-transitive}
Let $G$ and $H$ be two  connected graphs such that  $G_{SR}$ and $H_{SR}$ are vertex-transitive graphs. Then
$$dim_s(G\square H)=\min\{|\partial(G)|dim_s(H), |\partial(H)|dim_s(G)\}.$$
\end{theorem}

\begin{proof}
Since $G_{SR}$ and $H_{SR}$ are vertex-transitive graphs,  it follows from Lemma \ref{lemma-vertex-transitive} that $\beta(G_{SR}\times H_{SR})=\max\{|\partial(G)|\beta(H),|\partial(H)|\beta(G)\}$. So, by (\ref{eq_oellerman_gallai}) we have
\begin{align*}
dim_s(G\square H)&=|\partial(G)||\partial(H)|-\beta(G_{SR}\times H_{SR})\\
&=|\partial(G)||\partial(H)|-\max\{|\partial(G)|\beta(H_{SR}),|\partial(H)|\beta(G_{SR})\}\\
&=\min\{|\partial(G)|dim_s(H), |\partial(H)|dim_s(G)\}.
\end{align*}
\end{proof}

The observations prior to Corollary 8 and Theorem \ref{vertex-transitive} yield the following:

\begin{corollary}\label{corollaryTheoremtransitive}

\begin{enumerate}[{\rm (a)}]
\item  $dim_s(K_n\square K_r)=\min\{n(r-1), r(n-1)\}.$
\item  For any trees $T_1$ and $T_2$, $dim_s(T_1\square T_2)=\min\{l(T_1)(l(T_2)-1), l(T_2)(l(T_1)-1)\}.$
\item  \cite{Oellermann} $dim_s(C_{2n+1}\square C_{2r+1})=\min\{(2n+1)(r+1), (2r+1)(n+1)\}.$
\item  $dim_s(K_{n}\square C_{2r+1})=\min\{n(r+1), (2r+1)(n-1)\}.$
\item  $dim_s(T\square C_{2r+1})=\min\{l(T)(r+1), (2r+1)(l(T)-1)\}.$
\item  For any tree $T$, $dim_s(K_n\square T)=\min\{l(T)(n-1), n(l(T)-1)\}.$
\end{enumerate}
\end{corollary}

In order to obtain another consequence of Theorem \ref{vertex-transitive} we state another easily verified result.
\begin{lemma}\label{bound-dim-MMD}
For every graph $G$, $|\sigma(G)|-1\le dim_s(G)\le |\partial(G)|-1$.
\end{lemma}
\begin{proof}
Each simplicial vertex of $G$ is mutually maximally distant with every other extreme vertex of $G$. So, $G_{SR}$ has a subgraph isomorphic to $K_{|\sigma(G)|}$. Thus, $\alpha(G_{SR})\ge |\sigma(G)|-1$. Hence, by Theorem \ref{lem_oellerman}, $dim_s(G)\ge |\sigma(G)|-1$.

On the other hand, notice that for any graph $H$, $\alpha(H)\le |V(H)|-1$. Since $V(G_{SR})=\partial(G)$, it follows that $\alpha(G_{SR})\le |\partial(G)|-1$. Thus, by Theorem \ref{lem_oellerman}, the upper bound follows.
\end{proof}

Note that if $\sigma(G)=\partial(G)$, then by Lemma\ref{bound-dim-MMD} $dim_s(G)=|\partial(G)|-1$.
Hence, as a specific case of Theorem \ref{vertex-transitive} we obtain the following result.
\begin{corollary}\label{ej=cotasup}
Let $G$ and $H$ be two connected graphs.  If $\partial(G)=\sigma(G)$ and $\partial(H)=\sigma(H)$, then $$dim_s(G\square H)=\min\{|\partial(G)|(|\partial(H)|-1),|\partial(H)|(|\partial(G)|-1)\}.$$
\end{corollary}

The size of a largest matching is the \emph{matching number} of a graph $G$, and will be denoted
by $\mu(G)$. The following lemma will be useful to  establish  another  consequence of Theorem \ref{eq_oellerman_direct}.

\begin{lemma}\label{lemmaVertexcoverGxK2}
For any non-trivial  non-empty graphs $G$ and $H$,
$$\alpha(G\times H)\ge \mu(H)\alpha(G\times K_2)=\mu(H)\mu(G\times K_2)\ge 2\mu(G)\mu(H).$$
\end{lemma}

\begin{proof}
Let  $M=\{u_iv_i: \;i=1,...,k\}$, be a maximum matching of $H$. Let $A$ be a minimum vertex cover of $G\times H$. Let $A_i=A\cap(V(G)\times \{u_i,v_i\})$, $i\in\{1,,...,k\}$. Notice that $A_i\neq \emptyset$ for every $i\in\{1,...,k\}$. Since $A_i \cap A_j \ne \emptyset$ for $i\ne j$,  $|A_{1}|+|A_{2}|+...+|A_{k}|\le |A|$. Moreover, $A_i$ is a vertex cover of $G\times \langle\{u_i,v_i\}\rangle \cong G\times K_2$ for every $i\in \{1,...,k\}$. Hence $k\alpha(G\times K_2)\le \sum_{i=1}^{k}|A_{i}|\le|A|=\alpha(G\times H)$. As a result, $$\alpha(G\times H)\ge \mu(H)\alpha(G\times K_2).$$
Since $G\times K_2$ is a bipartite graph,  it follows from Theorem \ref{konig-egervary} that $$\alpha(G\times K_2)=\mu(G\times K_2).$$ Finally, every matching $\{x_iy_i: \;i\in \{1,...,k'\}\}$ of $G$ induces a matching $\{(x_i,a)(y_i,b),(y_i,a)(x_i,b): \;i\in\{1,...,k'\}\}$ of $G\times K_2$, where $\{a,b\}$ is the vertex set of $K_2$. Thus, $\mu(G\times K_2)\ge 2\mu(G)$. The completes the proof.
\end{proof}

Note that if $G$ and $H$ are two graphs (of order  $n_1$ and $n_2$, respectively) having a perfect matching and if at least one of them is bipartite, then $G\times H$ is bipartite and $\frac{n_1n_2}{2}\ge \mu(G\times H)=\alpha(G\times H)\ge 2\mu(G)\mu(H)=\frac{n_1n_2}{2}$. So, in this case we have $$\alpha(G\times H)= \mu(H)\alpha(G\times K_2)=\mu(H)\mu(G\times K_2)= 2\mu(G)\mu(H)=\frac{n_1n_2}{2}.$$

The following result is a direct consequence of Lemma \ref{lemmaVertexcoverGxK2} and Theorem \ref{eq_oellerman_direct}.

\begin{theorem}
Let $G$ and $H$ be two connected graphs.
$$dim_s(G\square H)\ge \mu(H_{SR})dim_s(G\square K_2)\ge 2\mu(G_{SR})\mu(H_{SR}).$$
\end{theorem}

Examples of graphs where  $dim_s(G\square H)= \mu(H_{SR})dim_s(G\square K_2)= 2\mu(G_{SR})\mu(H_{SR})=\frac{|\partial(G)||\partial(H)|}{2}$ are given in Corollary  \ref{corollaryTheoremMatching}.

\subsection{Strong metric dimension of Hamming graphs}

Now we study a particular case of Cartesian product graphs, the so called Hamming graphs. The Hamming graph $H_{k,n}$ is defined as the Cartesian product of $k$ copies of the complete graph $K_n$, {\em i.e.},
$$\begin{array}{c}
            H_{k,n}=\underbrace{K_n\;\Box\; K_n\;\Box\; ...\; \Box\; K_n} \\
            \;\;\;\;\;\;\;\;\;\;k \mbox{ times}
          \end{array}
$$

The strong metric dimension of Hamming graphs was obtained in \cite{strong-hamming} where the authors gave a long and complicated  proof.
 Here we give a simple proof for this result, using Theorem \ref{cartesian_directed} and the next result due to Valencia-Pabon and Vera \cite{valencia}.

\begin{lemma}\label{indep-direct-hamming}
For any positive integers, $n_1,n_2,...n_r$,
$$\beta(K_{n_1}\times K_{n_2}\times ... \times K_{n_r})=\max_{1\le i\le r}\left\{\frac{n_1n_2...n_r}{n_i}\right\}.$$
\end{lemma}

By Theorem \ref{cartesian_directed}, it follows that for any positive integers, $n_1,n_2,...n_r$, $$(K_{n_1}\Box K_{n_2}\Box ... \Box K_{n_r})_{SR}\cong K_{n_1}\times K_{n_2}\times ... \times K_{n_r}.$$ Therefore, equation (\ref{eq_oellerman_gallai}) and Lemma \ref{indep-direct-hamming} give the following result.

\begin{theorem}
For any positive integers, $n_1,n_2,...n_r$,
$$dim_s(K_{n_1}\square K_{n_2}\square ... \square K_{n_r})=n_1n_2...n_r-\max_{1\le i\le r}\left\{\frac{n_1n_2...n_r}{n_i}\right\}.$$
\end{theorem}
Hence, the above result gives as a corollary the value for the strong metric dimension of Hamming graphs.
\begin{corollary}
For any Hamming graph $H_{k,n}$, $dim_s(H_{k,n})=(n-1)n^{k-1}$.
\end{corollary}


\subsection{Relating the strong metric dimension of the Cartesian product of graphs to the strong metric dimension of its factor graphs}

\begin{lemma}\label{Lemma-Jha}{\em \cite{Jha}}
For any graphs $G$ and $H$ of orders $n_1$ and $n_2$, respectively,
$$\beta(G\times H)\ge \max\{ n_2\beta(G), n_1\beta(H)\}.$$
\end{lemma}
\begin{theorem}
For any connected graphs $G$ and $H$,
$$ dim_s(G\square H)\le \min \{dim_s(G)|\partial(H)|,|\partial(G)|dim_s(H)\}.$$
\end{theorem}

\begin{proof}
By Lemma \ref{Lemma-Jha},
$$
\beta(G_{SR}\times H_{SR})\ge \max\{  |\partial(H)|\beta (G_{SR}),  |\partial(G)|\beta (H_{SR})\}.
$$
Thus, by Gallai's theorem,
$$
\alpha(G_{SR}\times H_{SR})\le \min \{  |\partial(H)|\alpha (G_{SR}),  |\partial(G)|\alpha (H_{SR})\}.
$$
The result now follows from Theorem \ref{eq_oellerman_direct}.
\end{proof}

Several examples of graphs where $dim_s(G\square H)= \min \{dim_s(G)|\partial(H)|,|\partial(G)|dim_s(H)\}$ are given in Corollary \ref{corollaryTheoremtransitive}.

\begin{lemma}\label{cota-sup-indep-direct}\em \cite{spacapan}
For any graphs $G$ and $H$ of orders $n_1$ and $n_2$, respectively,
$$\beta(G\times H)\le n_2\beta(G)+n_1\beta(H)-\beta(G)\beta(H).$$
\end{lemma}

\begin{theorem}
For any connected graphs $G$ and $H$,
$$ dim_s(G\square H)\ge dim_s(G)dim_s(H).$$
\end{theorem}

\begin{proof}
By Lemma \ref{cota-sup-indep-direct},
$$\beta(G_{SR}\times H_{SR})\le |\partial(H)|\beta(G_{SR})+|\partial(G)|\beta(H_{SR})-\beta(G_{SR})\beta(H_{SR}).$$
Thus, by Gallai's theorem,
$$\alpha(G_{SR}\times H_{SR})\ge \alpha(G_{SR})\alpha( H_{SR}).$$
The result now follows from Theorem \ref{eq_oellerman_direct}.
\end{proof}

\subsection{Cartesian product graphs with strong metric dimension 2}

\begin{lemma}\label{lem_path}
For every connected graph $G$ of order $n$, $dim_s(G) = 1$ if and only if $G\cong P_n$.
\end{lemma}

\begin{proof}
If $G\cong P_n$, then $dim_s(G) = 1$. Since $1 \le dim_(G) \le dim_s(G)$ it follows, if $dim_s(G) = 1$, that $dim(G) =1$. Thus $G$ is a path.
\end{proof}

\begin{lemma}\label{LemmaIntervals} Given $a,x,c\in V(G)$ and $b,y,d\in V(H)$,
$(a,b)\in I_{G\Box H}[(x,y),(c,d)]$ if and only if $a\in I_{G}[x,c]$ and $b\in I_{H}[y,d]$.
\end{lemma}
\begin{proof}
If $a\in I_{G}[x,c]$, then $d_{G}(x,c)=d_{G}(x,a)+d_{G}(a,c)$. Similarly if $b\in I_{H}[y,d]$, then
$d_{H}(y,d)=d_{H}(y,b)+d_{H}(b,d)$. Hence
\begin{align*}d_{G\Box H}((x,y),(c,d))&=d_{G}(x,c)+d_{H}(y,d)\\
&=(d_{G}(x,a)+d_{G}(a,c))+(d_{H}(y,b)+d_{H}(b,d))\\
&=(d_{G}(x,a)+d_{H}(y,b))+(d_{G}(a,c)+d_{H}(b,d))\\
&=d_{G\Box H}((x,y),(a,b))+d_{G\Box H}((a,b),(c,d)).
\end{align*}
Thus $(a,b)\in I_{G\Box H}[(x,y),(c,d)]$.

Conversely,  if $(a,b)\in I_{G\Box H}[(x,y),(c,d)]$, then
\begin{align*}d_{G\Box H}((x,y),(c,d))&=d_{G\Box H}((x,y),(a,b))+d_{G\Box H}((a,b),(c,d))\\
&=(d_{G}(x,a)+d_{H}(y,b))+(d_{G}(a,c)+d_{H}(b,d))\\
&=(d_{G}(x,a)+d_{G}(a,c))+(d_{H}(y,b)+d_{H}(b,d)).
\end{align*}
Now, if  $a\not\in I_{G}[x,c]$ or $b\not\in I_{G}[y,d]$, then
$d_{G\Box H}((x,y),(c,d))> d_{G}(x,c)+d_{H}(y,d),$ a contradiction.
\end{proof}

\begin{proposition}
Let $G$ and $H$ be two connected graphs of order at least $2$. Then $dim_s(G\square H) = 2$ if and only if $G$ and $H$ are both paths.
\end{proposition}

\begin{proof}
If $G$ and $H$ are paths, then, by Theorem \ref{eq_oellerman_direct},
\begin{align*}
dim_s(G\square H)&=\alpha (K_2\times K_2)=2.
\end{align*}
On the other hand, let $S = \{(a,x),(b,y)\}$ be a strong metric basis of $G\square H$. If $a\neq b$ and $x\neq y$. Let $c$ be a neighbor of $b$ on a $a-b$ path (possibly $a = c$). Let $z$ be a neighbor of $y$ on a $x-y$ path (notice that could be $x = z$). So,  we have $d_{G\square H}((b,z),(a,x)) = d_G(a,b)+d_H(z,x) = d_G(a,c)+1+d_H(x,y)-1 = d_G(a,c)+d_H(x,y) = d_{G\square H}((c,y),(a,x))$. Thus, $(b,z)\notin I_{G\square H}[(c,y),(a,x)]$ and $(c,y)\notin I_{G\square H}[(b,z),(a,x)]$. Moreover, $d_{G\square H}((b,z),(b,y)) = d_H(z,y) = 1 = d_G(a,c) = d_{G\square H}((c,y),(b,y))$. Thus, $(b,z)\notin I_{G\square H}[(c,y),(b,y)]$ and $(c,y)\notin I_{G\square H}[(b,z),(b,y)]$. Therefore, $S = \{(a,x),(b,y)\}$ does not strongly resolve $(b,z)$ and $(c,y)$.

If $a = b$, then the projection of $S$ onto $G$ is a single vertex. By Lemma \ref{LemmaIntervals}, the projection of $S$ onto $G$ strongly resolves $G$ and thus, by Lemma \ref{lem_path}, $G$ is a path. Similarly, if $x = y$, then $H$ is a path.  Therefore, $G$ or $H$ is a path. We assume, without loss of generality, that $G$ is a path. By Theorem \ref{eq_oellerman_direct} it follows that $2=dim_s(G\square H)=\alpha(K_2\times H_{SR})$. Thus, either $H_{SR}$ is isomorphic to $K_2$ or $\alpha(H_{SR})=1$. Thus $dim_s(H)=1$. Therefore, by Lemma \ref{lem_path}, $H$ is a path.
\end{proof}

\section{Strong metric dimension of direct product graphs}

\begin{lemma}{\em \cite{direct-cart-isom}}
Let $G$ and $H$ be two connected graphs. Then $G\Box H\cong G\times H$ if and only if $G\cong H\cong C_{2k+1}$ for some positive integer $k$.
\end{lemma}

The above lemma and Corollary \ref{corollaryTheoremtransitive} (c) give the following result.

\begin{proposition}
For any positive integer $k$, $dim_s(C_{2k+1}\times C_{2k+1})=(2k+1)(k+1)$.
\end{proposition}

\begin{lemma}\label{k_r-times-k_t}
For any positive integers $r,t\ge 3$,
 $(K_r\times K_t)_{SR}\cong K_r\Box K_t$.
\end{lemma}

\begin{proof}
Let $V_1$ and $V_2$ be the vertex sets of $K_r$ and $K_t$, respectively. Let $(u_1,v_1)$ and $(u_2,v_2)$ be two distinct vertices of $K_r\times K_t$. If $u_1=u_2$ or $v_1=v_2$, then $d_{K_r\times K_t}((u_1,v_1),(u_2,v_2))=2$. On the other hand, if $u_1\ne u_2$ and $v_1\ne v_2$, then $d_{K_r\times K_t}((u_1,v_1),(u_2,v_2))=1$. Thus, any two distinct vertices $(u_1,v_1)$ and $(u_2,v_2)$ are mutually maximally distant in $K_r\times K_t$ if and only if either $u_1=u_2$ or $v_1=v_2$. So, every vertex $(x,y)$ is adjacent in $(K_r\times K_t)_{SR}$ to all the vertices of the sets $\{(x,v_i)\,:\,v_i\in V_2-\{y\}\}$ and $\{(u_i,y)\,:\,u_i\in V_1-\{x\}\}$ and thus, $(K_r\times K_t)_{SR}$ is isomorphic to the Cartesian product graph $K_r\Box K_t$.
\end{proof}

\begin{lemma}{\rm (\cite{Vizing1963}, Vizing, 1963)} \label{BizingIndep}For any graphs $G$ and $H$ of order $r$ and $t$, respectively,
$$\beta(G)\beta(H)+\min\{r-\beta(G),t-\beta(H)\}\le \beta(G\Box H)\le \min\{t\beta(G),r\beta(H).\}$$
\end{lemma}

\begin{proposition}\label{propoCartDirecComplete}
For any positive integers $r,t\ge 3$,
 $dim_s(K_r\times K_t)=\max\{r(t-1),t(r-1)\}$
\end{proposition}

 \begin{proof}
 By Theorem \ref{lem_oellerman}, Lemma \ref{k_r-times-k_t}, and Gallai's theorem
 $dim_s(K_r\times K_t)=rt-\beta(K_r\Box K_t)$. By Lemma \ref{BizingIndep}, $\beta(K_r\Box K_t)=\min\{r,t\}$. Thus $dim_s(K_r\times K_t)=rt-\min\{r,t\}=\max\{r(t-1),t(r-1)\}$.
\end{proof}

We now introduce a well-known class of graphs that will be used to  prove our next result. Let $\mathbb{Z}_n$ be the additive group of integers modulo $n$ and let $M\subset \mathbb{Z}_n$, such that, $i\in M$ if and only if $-i\in M$. We can construct a graph $G=(V,E)$ as follows: the vertices of $V$ are the elements of $\mathbb{Z}_n$ and $(i,j)$ is an edge in $E$ if and only if $j-i\in M$. This graph is a \emph{circulant of order $n$} and is denoted by $CR(n, M)$. With this notation, a cycle is  the same as $CR(n, \{-1,1\})$ and the complete graph is $CR(n,\mathbb{Z}_n)$. In order to simplify the notation
we will use $CR(n,t)$, $0<t\leq\frac{n}{2}$, instead of $CR(n,\{-t,-t+1,...,-1,1,2,...,t\})$. This is also the $t^{th}$ power of $C_n$.

\begin{lemma}\label{circulant-indep}
For any circulant graph $CR(n,2)$, $\beta(CR(n,2))=\left\lfloor\frac{n}{3}\right\rfloor$.
\end{lemma}

\begin{proof}
Let $V=\{u_0,u_1,...,u_{n-1}\}$ be the set of vertices of $CR(n,2)$, where two vertices $u_i,u_j$ are adjacent if and only if $i-j\in \{-2,-1,1,2\}$. Notice that every vertex $u_i$ is adjacent to the vertices $u_{i-2},u_{i-1},u_{i+1},u_{i+2}$, where the operations with the subindex $i$ are expressed modulo $n$. Let $S$ be the set of vertices of $CR(n,2)$ satisfying the following.
\begin{itemize}
\item If $n\equiv 0$ mod $3$, then $S=\{u_0,u_3,u_6,...,u_{n-6},u_{n-3}\}$.
\item If $n\equiv 1$ mod $3$, then $S=\{u_0,u_3,u_6,...,u_{n-7},u_{n-4}\}$.
\item If $n\equiv 2$ mod $3$, then $S=\{u_0,u_3,u_6,...,u_{n-8},u_{n-5}\}$.
\end{itemize}
Notice that $S$ is an independent set. Thus, $\beta(CR(n,2))\ge |S|=\left\lfloor\frac{n}{3}\right\rfloor$. Now, let us suppose that $\beta(CR(n,2))>\left\lfloor\frac{n}{3}\right\rfloor$ and let $S'$ be an independent set of maximum cardinality in $CR(n,2)$. Hence there exist two vertices $u_i,u_j\in S'$ such that either $i=j+1$, $i=j-1$, $i=j+2$ or $i=j-2$, where the operations with the subindexes $i,j$ are expressed modulo $n$. Thus, $i-j\in \{-2,-1,1,2\}$ and, hence, $u_i$ and $u_j$ are adjacent, which is a contradiction. Therefore, $\beta(CR(n,2))=\left\lfloor\frac{n}{3}\right\rfloor$ and the proof is complete.
\end{proof}

From now on we will use the notation $u\sim v$ if $u$ and $v$ are adjacent vertices. For a  vertex
$v$ of a graph $G$, $N_G(v)$ will denote the set of neighbors of $v$ in $G$, i.e.,  $N_G(v)=\{u\in V:\; u\sim v\}$.

\begin{theorem}
For any positive integers $r\ge 4$ and $t\ge 3$,
$$dim_s(C_r\times K_t)=\left\{\begin{array}{ll}
                                t(r-1), & \mbox{ if $r\in \{4,5\}$,} \\
                                & \\
                                \frac{tr}{2}, & \mbox{ if $r$ is even and $r\ge 6$,} \\
                                & \\
                                t(r-\left\lfloor\frac{r}{3}\right\rfloor), & \mbox{ otherwise.}
                              \end{array}
\right.$$
\end{theorem}

\begin{proof}
Let $V_1=\{u_0,u_1,...,u_{r-1}\}$ and $V_2=\{v_1,v_2,....,v_{t}\}$ be the vertex sets of $C_r$ and $K_t$, respectively. We assume $C_r=u_0 u_1 u_2 \cdots u_{r-1} u_0$ in $C_r$. Hereafter all the operations with the subindex of a vertex $u_i$ of $C_r$ are expressed modulo $r$. Let $(u_i,v_j),(u_l,v_k)$ be two distinct vertices of $C_r\times K_t$.

\noindent{Case 1:} Let $r=4$ or $5$.

Subcase 1.1: $u_i=u_l$. Hence, $d_{C_r\times K_t}((u_i,v_j),(u_l,v_k))=2$. Since $(u_i,v_j)\sim (u_{i-1},v_k)$ if $k \ne j$and $d_{C_r\times K_t}((u_{i-1},v_k),(u_l,v_k))=3$ it follows that  $(u_i,v_j)$ and $(u_l,v_k)$ are not mutually maximally distant in $C_r\times K_t$.

Subcase 1.2: $v_j=v_k$. If $l=i+1$ or $i=l+1$, then without loss of generality we suppose $l=i+1$ and we have that $d_{C_r\times K_t}((u_i,v_j),(u_l,v_k))=3=D(C_r\times K_t)$. Thus, $(u_i,v_j)$ and $(u_l,v_k)$ are mutually maximally distant in $C_r\times K_t$. On the other hand, if $l\ne i+1$ and $i\ne l+1$, then $d_{C_r\times K_t}((u_i,v_j),(u_l,v_k))=2$. Since for every vertex $(u,v)\in N_{C_r\times K_t}(u_i,v_j)$ we have that $d_{C_r\times K_t}((u,v),(u_l,v_k))\le 2$ and also for every vertex $(u,v)\in N_{C_r\times K_t}(u_l,v_k)$ we have that $d_{C_r\times K_t}((u,v),(u_i,v_j))\le 2$, we obtain that $(u_i,v_j)$ and $(u_l,v_k)$ are mutually maximally distant in $C_r\times K_t$.

Subcase 1.3: $u_i\ne u_l$, $v_j\ne v_k$ and $(u_i,v_j)\sim (u_l,v_k)$. So, there exists a vertex $(u,v)\in N_{C_r\times K_t}(u_l,v_k)$ such that $d_{C_r\times K_t}((u,v),(u_i,v_j))=2$ and, as a consequence, $(u_i,v_j)$ and $(u_l,v_k)$ are not mutually maximally distant in $C_r\times K_t$.

Subcase 1.4: $u_i\ne u_l$, $v_j\ne v_k$ and $(u_i,v_j)\not\sim (u_l,v_k)$. Hence, $d_{C_r\times K_t}((u_i,v_j),(u_l,v_k))=2$. We can suppose, without loss of generality, that $l=i+2$. Since ($(u_i,v_j)\sim (u_{l-1},v_k)$ and $(u_l,v_k)\sim (u_{l-1},v_j)$) and also ($d_{C_r\times K_t}((u_i,v_j),(u_{l-1},v_j))=3$ and $d_{C_r\times K_t}((u_l,v_k),(u_{l-1},v_k)=3$), we obtain that $(u_i,v_j)$ and $(u_l,v_k)$ are not mutually maximally distant in $C_r\times K_t$.

Hence the strong resolving graph $(C_r\times K_t)_{SR}$ is isomorphic to $\bigcup_{i=1}^t K_r$. Thus, by Theorem \ref{lem_oellerman},

$$dim_s(C_r\times K_t)=\alpha((C_r\times K_t)_{SR})=\alpha\left(\bigcup_{i=1}^t K_r\right)=\sum_{i=1}^{t}\alpha(K_r)=t(r-1).$$

\noindent{Case 2:} $r\ge 6$. Let $(u_i,v_j),(u_l,v_k)$ be two different vertices of $C_r\times K_t$.

Subcase 2.1: $u_i=u_l$. As in Subcase 1.1 it can be shown that $(u_i,v_j),(u_l,v_k)$ are not mutually maximally distant.

Subcase 2.2: $v_j=v_k$. We consider the following further subcases.
\begin{enumerate}[(a)]
\item $l=i+1$ or $i=l+1$. Without loss of generality we suppose $l=i+1$. Hence, $d_{C_r\times K_t}((u_i,v_j),(u_l,v_k))=3$. Notice that $N_{C_r\times K_t}(u_i,v_j)=\{u_{i-1},u_{i+1}\}\times (V_2-\{v_j\})$ and $N_{C_r\times K_t}(u_l,v_k)=\{u_{i},u_{i+2}\}\times (V_2-\{v_k\})$. Thus, for every vertex $(u,v)\in N_{C_r\times K_t}(u_i,v_j)$ it follows that $d_{C_r\times K_t}((u,v),(u_l,v_k))\le 2$ and for every vertex $(u,v)\in N_{C_r\times K_t}(u_l,v_k)$ it follows that $d_{C_r\times K_t}((u,v),(u_i,v_j))\le 2$. Hence, $(u_i,v_j)$ and $(u_l,v_k)$ are mutually maximally distant in $C_r\times K_t$.

\item $l\ne i+1$, $i\ne l+1$ and $d_{C_r}(u_i,u_l)<D(C_r)$. So, $d_{C_r\times K_t}((u_i,v_j),(u_l,v_k))=\min\{l-i,i-l\}$. Since $(u_i,v_j)\sim (u_{i-1},v_q)$ with $q\ne j$ and $d_{C_r\times K_t}((u_{i-1},v_q),(u_l,v_k))=\min\{l-i+1,i-l+1\}$ we have that $(u_i,v_j)$ and $(u_l,v_k)$ are not mutually maximally distant in $C_r\times K_t$.
\item $l\ne i+1$, $i\ne l+1$ and $d_{C_r}(u_i,u_l)=D(C_r)$. Thus, $d_{C_r\times K_t}((u_i,v_j),(u_l,v_k))=\min\{l-i,i-l\}=D(C_r)=\left\lfloor\frac{r}{2}\right\rfloor$ and, as a consequence, we have that $(u_i,v_j)$ and $(u_l,v_k)$ are mutually maximally distant in $C_r\times K_t$.
\end{enumerate}

Subcase 2.3: $u_i\ne u_l$, $v_j\ne v_k$ and $d_{C_r}(u_i,u_l)<D(C_r)$. As in Subcase 2.2(b) it can be shown that $(u_i,v_j)$ and $(u_l,v_k)$ are not mutually maximally distant in $C_r\times K_t$.

Subcase 2.4: $u_i\ne u_l$, $v_j\ne v_k$ and $d_{C_r}(u_i,u_l)=D(C_r)$. As in Subcase 2.2(c) it can be shown that $(u_i,v_j)$ and $(u_l,v_k)$ are mutually maximally distant in $C_r\times K_t$.

From the above cases it follows that the strong resolving graph $(C_r\times K_t)_{SR}$ has vertex set $V_1\times V_2$ and two vertices $(u_i,v_j),(u_l,v_k)$ are adjacent in this graph if and only if either, ($\min\{l-i,i-l\}=1$ and $j=k$) or ($\min\{l-i,i-l\}=D(C_r)=\left\lfloor\frac{r}{2}\right\rfloor$ and $1\le j,k\le t$). Next we obtain the vertex cover number of $(C_r\times K_t)_{SR}$.

If $r$ is even, then every vertex $(u_i,v_j)$ has $t$ neighbors of type $(u_{i+r/2},v_l)$, $1\le l\le t$ and two neighbors $(u_{i-1},v_j),(u_{i+1},v_j)$. So, $\alpha((C_r\times K_t)_{SR})\ge t\alpha(C_r)=t\frac{r}{2}$. On the other hand, if we take the set of vertices $A=\{(u_i,v_j)\,:\,i\in \{0,2,4,...,r-2\},\;j\in \{1,...,t\}\}$, then every edge of $(C_r\times K_t)_{SR}$ is incident to some vertex of $A$. So, $A$ is a vertex cover and $\alpha((C_r\times K_t)_{SR})\le |A|= t\frac{r}{2}$. Hence $\alpha((C_r\times K_t)_{SR})=t\frac{r}{2}$. Therefore
$$dim_s(C_r\times K_t)=\alpha((C_r\times K_t)_{SR})=t\frac{r}{2}.$$

If $r$ is odd, then every vertex $(u_i,v_j)$ has $t$ neighbors of type $(u_{i+(r-1)/2},v_l)$, $t$ neighbors of type $(u_{i+r/2},v_l)$, $1\le l\le t$, and the two neighbors $(u_{i-1},v_j),(u_{i+1},v_j)$. Thus for every $k\in \{1,...,t\}$ it follows that
\begin{equation}\label{adjacencies}
(u_{0},v_k)\sim (u_{\frac{r-1}{2}},v_k)\sim (u_{r-1},v_k)\sim (u_{\frac{r-1}{2}-1},v_k)\sim (u_{r-2},v_k)\sim \cdots \sim (u_{1},v_k)\sim (u_{\frac{r+1}{2}},v_k)\sim (u_{0},v_k).
\end{equation}
Also, since $(u_{0},v_k)\sim (u_{1},v_k)\sim \cdots \sim (u_{r-1},v_k)\sim (u_{0},v_k)$, the graph $G'$ formed from $t$ disjoint copies of a circulant graph $CR(r,2)$ is a subgraph of $(C_r\times K_t)_{SR}$. By Lemma \ref{circulant-indep}
$$\alpha((C_r\times K_t)_{SR})\ge t\alpha(CR(r,2))=t(r-\beta(CR(r,2))=t\left(r-\left\lfloor\frac{r}{3}\right\rfloor\right).$$
Now, we will rename the vertices of $C_r$ according to the adjacencies in (\ref{adjacencies}), {\em i.e.}, $u'_0=u_0$, $u'_1=u_{\frac{r-1}{2}}$, $u'_2=u_{r-1}$, $u'_3=u_{\frac{r-1}{2}-1}$, $u'_2=u_{r-2}$, ..., $u'_{r-2}=u_1$ and $u'_{r-1}=u_{\frac{r+1}{2}}$. With this notation, we define a set $B$, of vertices of $(C_r\times K_t)_{SR}$, as follows:
\begin{itemize}
\item $B=\{(u'_i,v_j):\,i\in \{0,1,3,4,6,7,...,r-3,r-2\},\;j\in \{1,...,t\}\}$, if $r\equiv 0$ (mod $3$).
\item $B=\{(u'_i,v_j):\,i\in \{0,1,3,4,6,7,...,r-4,r-3,r-1\},\;j\in \{1,...,t\}\}$, if $r\equiv 1$ (mod $3$).
\item $B=\{(u'_i,v_j):\,i\in \{0,1,3,4,6,7,...,r-5,r-4,r-2,r-1\},\;j\in \{1,...,t\}\}$, if $r\equiv 2$ (mod $3$)
\end{itemize}
Note that if $(u,v),(x,y)\notin B$, then $(u,v)\not\sim (x,y)$ and, thus $B$ is a vertex cover of $(C_r\times K_t)_{SR}$. Hence,
$\alpha((C_r\times K_t)_{SR})\le |B|=t(r-\left\lfloor\frac{r}{3}\right\rfloor)$, which leads to $\alpha((C_r\times K_t)_{SR})=t(r-\left\lfloor\frac{r}{3}\right\rfloor)$.
Therefore, we have the following
$$dim_s(C_r\times K_t)=\alpha((C_r\times K_t)_{SR})=t\left(r-\left\lfloor\frac{r}{3}\right\rfloor\right).$$
\end{proof}

\begin{theorem}
For any positive integers $r\ge 2$ and $t\ge 3$,
$$dim_s(P_r\times K_t) = t\left \lceil \frac{r}{2}\right \rceil.$$
\end{theorem}

\begin{proof}
Let $V_1=\{u_1,u_2,...,u_r\}$ and $V_2=\{v_1,v_2,....,v_t\}$ be the vertex sets of $P_r$ and $K_t$, respectively. We assume $u_1\sim u_2\sim u_3\sim \cdots \sim u_r$ in $P_r$. If $r=2$, then a vertex $(u_i,v_j)$ in $P_2\times K_t$ is mutually maximally distant only with the vertex $(u_l, v_j)$, where $i\ne l$. So, $(P_2\times K_t)_{SR}\cong \bigcup_{m=1}^{t} K_2$. Thus, by Theorem \ref{lem_oellerman},
$$dim_s(P_2\times K_t)=\alpha((P_2\times K_t)_{SR})=\alpha\left(\bigcup_{i=1}^{t} K_2\right)=\sum_{i=1}^t\alpha(K_2)=t.$$

If $r=3$, then a vertex $(u_i,v_j)$ in $P_3\times K_t$ is mutually maximally distant only with those vertices $(u_l, v_j)$, where $i\ne l$. Thus, $(P_3\times K_t)_{SR}\cong \bigcup_{m=1}^{t} K_3$ and, by Theorem \ref{lem_oellerman},
$$dim_s(P_3\times K_t)=\alpha((P_3\times K_t)_{SR})=\alpha \left(\bigcup_{i=1}^{t} K_3\right)=\sum_{i=1}^t\alpha(K_3)=t\left \lceil \frac{r}{2}\right \rceil.$$

From now on we suppose $r\ge 4$. Let $(u_i,v_j),(u_l,v_k)$ be two different vertices of $P_r\times K_t$. We consider the following cases.

Case 1: $u_i=u_l$. Hence, it is satisfied that $d_{P_r\times K_t}((u_i,v_j),(u_l,v_k))=2$. If $i\ne 1$, then $(u_i,v_j)\sim (u_{i-1},v_k)$ and $d_{P_r\times K_t}((u_{i-1},v_k),(u_l,v_k))=3$. Also, if $i=1$, then $(u_i,v_j)\sim (u_{i+1},v_k)$ and $d_{P_r\times K_t}((u_{i+1},v_k),(u_l,v_k))=3$. Thus, $(u_i,v_j)$ and $(u_l,v_k)$ are not mutually maximally distant in $P_r\times K_t$.

Case 2: $v_j=v_k$ and, without loss of generality,  $i<l$. We have the following cases.
\begin{enumerate}[(a)]

\item If $u_i\sim u_l$ in $P_r$, then $d_{P_r\times K_t}((u_i,v_j),(u_l,v_k))=3$. Let $(u_a,v_b)$ be a vertex such that $(u_i,v_j)\sim (u_a,v_b)$. So, ($a=i-1$ or $a=l$) and $b\ne j$. Thus, for every $(u_a,v_b)$ we have that $d_{P_r\times K_t}((u_a,v_b),(u_l,v_k))=2<3=d_{P_r\times K_t}((u_i,v_j),(u_l,v_k))$. Now, let $(u_c,v_d)$ be a vertex such that $(u_l,v_k)\sim (u_c,v_d)$. So, ($c=i$ or $c=l+1$) and $d\ne j$. Thus, for every $(u_c,v_d)$ we have that $d_{P_r\times K_t}((u_c,v_d),(u_i,v_j))=2<3=d_{P_r\times K_t}((u_i,v_j),(u_l,v_k))$. Therefore, $(u_i,v_j)$ and $(u_l,v_k)$ are mutually maximally distant in $P_r\times K_t$.

\item If $u_i\not\sim u_l$ in $P_r$, then $d_{P_r\times K_t}((u_i,v_j),(u_l,v_k))=|i-l|$. Now, if $u_i\ne u_1$, then for every vertex $(u_{i-1},v_p)$, $p\ne j$, we have that $(u_i,v_j)\sim (u_{i-1},v_p)$ and $d_{P_r\times K_t}((u_{i-1},v_p),(u_l,v_k))=|i-l+1|$. Similarly, if $u_l\not \sim u_r$, then for every vertex $(u_{l+1},v_p)$, $p\ne j$, we have that $(u_l,v_k)\sim (u_{l+1},v_p)$ and $d_{P_r\times K_t}((u_{l+1},v_p),(u_i,v_j))=|i-l+1|$. Thus, we obtain that $(u_i,v_j)$ and $(u_l,v_k)$ are not mutually maximally distant in $P_r\times K_t$.

\item If $u_i=u_1$ and $u_l=u_r$, then $d_{P_r\times K_t}((u_i,v_j),(u_l,v_k))=r-1=D(P_r\times K_t)$. Thus, $(u_i,v_j)$ and $(u_l,v_k)$ are mutually maximally distant in $P_r\times K_t$.
\end{enumerate}

Case 3: $u_i\ne u_l$, $v_j\ne v_k$ and we consider, without loss of generality,  $i<l$. If $u_i\ne u_1$ or $u_l\ne u_r$, then as in Case 2 (b) it follows that $(u_i,v_j)$ and $(u_l,v_k)$ are not mutually maximally distant in $P_r\times K_t$. On the other hand, if $u_i=u_1$ and $u_l=u_r$, then as in Case 2 (c) it follows that $(u_i,v_j)$ and $(u_l,v_k)$ are mutually maximally distant in $P_r\times K_t$.

Therefore, $(P_r\times K_t)_{SR}$ is isomorphic to a graph with vertex set $V_1\times V_2$ and such that two vertices $(u_i,v_j),(u_l,v_k)$ are adjacent if and only if either, ($|l-i|=1$ and $j=k$) or ($|l-i|=r-1$ and $1\le j,k\le r$). Notice that every vertex $(u_i,v_j)$, where $1<i<r$, has only two neighbors $(u_{i-1},v_j)$ and $(u_{i+1},v_j)$, while every vertex $(u_1,u_j)$ has a neighbor $(u_2,u_j)$ and $r$ neighbors of type $(u_r,u_l)$, $1\le l\le t$. Also, every vertex $(u_r,u_j)$ has a neighbor $(u_{r-1},u_j)$ and $r$ neighbors of type $(u_1,u_l)$, $1\le l\le t$. So, $(P_r\times K_t)_{SR}$ has a subgraph $G'$ isomorphic to the disjoint union of $t$ cycles of order $r$ and, as a consequence, $\alpha((P_r\times K_t)_{SR})\ge t\alpha(C_r)=t\left \lceil \frac{r}{2}\right \rceil$.

On the other hand, let $r$ be an even number. If we take the set of vertices $A=\{(u_i,v_j)\,:\,i\in \{1,3,5,...,r-1\},\;j\in \{1,...,t\}\}$, then every edge of $(P_r\times K_t)_{SR}$ is incident to some vertex of $A$. Thus, $A$ is a vertex cover of $(P_r\times K_t)_{SR}$ and we have that $\alpha((P_r\times K_t)_{SR})\le |A|=t\left \lceil \frac{r}{2}\right \rceil$. Now, suppose $r$ odd. If we take the set of vertices $B=\{(u_i,v_j)\,:\,i\in \{1,3,5,...,r\},\;j\in \{1,...,t\}\}$, then every edge of $(P_r\times K_t)_{SR}$ is incident to some vertex of $B$. So, $B$ is a vertex cover of $(P_r\times K_t)_{SR}$ and thus $\alpha((P_r\times K_t)_{SR})\le |B|=t\left \lceil \frac{r}{2}\right \rceil$. Hence $\alpha((P_r\times K_t)_{SR})=t\left \lceil \frac{r}{2}\right \rceil$. Therefore, from Theorem \ref{lem_oellerman},
$$dim_s(P_r\times K_t)=\alpha((P_r\times K_t)_{SR})=t\left \lceil \frac{r}{2}\right \rceil.$$
\end{proof}


\begin{thebibliography}{99}

\bibitem{bres-klav-tepeh} B. Bre\v{s}ar, S. Klav\v{z}ar, and A. Tepeh Horvat, On the geodetic number and related metric sets in Cartesian product graphs, {\it Discrete Mathematics} {\bf 308} (2008) 5555--5561.

\bibitem{brigham} R. C. Brigham, G. Chartrand, R. D. Dutton, and P. Zhang, Resolving domination in graphs, {\it Mathematica Bohemica} {\bf 128} (1) (2003) 25--36.


\bibitem{pelayo1} J. C\'aceres, C. Hernando, M. Mora, I. M. Pelayo, M. L. Puertas, C. Seara, and D. R. Wood, On the metric dimension of Cartesian product of graphs, {\it SIAM Journal on Discrete Mathematics} {\bf 21} (2) (2007) 273--302.


\bibitem{chappell} G. Chappell, J. Gimbel, and C. Hartman, Bounds on the metric and partition dimensions of a graph, {\it Ars Combinatoria} {\bf 88} (2008) 349--366.

\bibitem{chartrand} G. Chartrand, L. Eroh, M. A. Johnson, and O. R. Oellermann, Resolvability in graphs and the metric dimension of a graph, {\it Discrete Applied Mathematics} {\bf 105} (2000) 99--113.

\bibitem{chartrand1} G. Chartrand, C. Poisson, and P. Zhang, Resolvability and the upper dimension of graphs, {\it Computers and Mathematics with Applications} {\bf 39} (2000) 19--28.

\bibitem{chartrand2} G. Chartrand, E. Salehi, and P. Zhang, The partition dimension of a graph, {\it  Aequationes Mathematicae} (1-2) \textbf{59} (2000) 45--54.

\bibitem{fehr} M. Fehr, S. Gosselin, and O. R. Oellermann, The partition dimension of Cayley digraphs, {\it Aequationes Mathematicae} {\bf 71} (2006) 1--18.

\bibitem{harary} F. Harary and R. A. Melter, On the metric dimension of a graph, {\it Ars Combinatoria} {\bf 2} (1976) 191--195.

\bibitem{haynes} T. W. Haynes, M. Henning, and J. Howard, Locating and total dominating sets in trees, {\it Discrete Applied Mathematics} {\bf 154} (2006) 1293--1300.


\bibitem{Jha} P. K. Jha and G. Slutzki, Independence numbers of product graphs,
\emph{Applied Mathematics Letters }\textbf{7} (4) (1994) 91--94.

\bibitem{pharmacy1} M. A. Johnson, Structure-activity maps for visualizing the graph variables arising in drug design, {\it Journal of Biopharmaceutical Statistics} {\bf 3} (1993) 203--236.

\bibitem{pharmacy2} M. A. Johnson, Browsable structure-activity datasets, {\it Advances in Molecular Similarity} (R. Carb\'o--Dorca and P. Mezey, eds.) JAI Press Connecticut (1998) 153--170.

\bibitem{landmarks} S. Khuller, B. Raghavachari, and A. Rosenfeld, Landmarks in graphs, {\it Discrete Applied Mathematics} {\bf 70} (1996) 217--229.

\bibitem{strong-hamming} J. Kratica, V. Kova\v{c}evi\'c-Vuj\v{c}i\'c, M. \v{C}angalovi\'c, and M. Stojanovi\'c, Minimal doubly resolving sets and the strong metric dimension of Hamming graphs, \emph{Applicable Analysis and Discrete Mathematics} \textbf{6} (2012) 63--71.


\bibitem{may-oellermann} T. R. May and O.R. Oellermann, The strong dimension of distance-hereditary graphs, {\em JCMCC} {\bf 76} (2011) 59--73.

\bibitem{Tomescu1} R. A. Melter and  I. Tomescu, Metric bases in digital geometry, {\it Computer Vision  Graphics and Image Processing} {\bf 25} (1984) 113--121.

\bibitem{direct-cart-isom} D. J. Miller, The categorical product of graphs, {\em Canadian Journal of Mathematics} {\bf 20} (1968) 1511--1521.

\bibitem{Oellermann} O. R. Oellermann and  J. Peters-Fransen, The strong metric dimension of graphs and digraphs, {\it Discrete Applied Mathematics} {\bf 155} (2007) 356--364.

\bibitem{LocalMetric} F. Okamoto,   B. Phinezyn, and  P. Zhang,
The local metric dimension of a graph. \emph{ Mathematica Bohemica}  \textbf{135} (3) (2010)   239--255.

\bibitem{survey} V. Saenpholphat and P. Zhang, Conditional resolvability in graphs: a survey, {\it International Journal of Mathematics and Mathematical Sciences} {\bf 38} (2004) 1997--2017.

\bibitem{seb} A. Seb\H{o} and E. Tannier, On metric generators of graphs, {\it Mathematics of Operations Research} {\bf 29} (2) (2004) 383--393.

\bibitem{leaves-trees} P. J. Slater, Leaves of trees, Proceeding of the 6th Southeastern Conference on Combinatorics, Graph Theory, and Computing, {\it Congressus Numerantium} {\bf 14} (1975) 549--559.

\bibitem{slater2} P.J. Slater, Dominating and reference sets in a graph, {\it Journal of Mathematical and Physical Sciences} {\bf 22} (4) (1988) 445--455.

\bibitem{spacapan} S. \v{S}pacapan, The $k$-independence number of direct products of graphs and Hedetniemi's conjecture, \emph{European Journal of Combinatorics} \textbf{32} (2011) 1377--1383.

\bibitem{tomescu} I. Tomescu, Discrepancies between metric and partition dimension of a connected graph, {\it Discrete Mathematics} {\bf 308} (2008) 5026--5031.

\bibitem{valencia} M. Valencia-Pabon and J. Vera, Independence and coloring properties of direct products of some vertex-transitive graphs, {\it Discrete Mathematics} {\bf 306} (2006) 2275--2281.

\bibitem{Vizing1963} V. G. Vizing, The Cartesian product of graphs (Russian), \emph{Vy\v{c}islitel'nye Sistemy} \textbf{9} (1963) 30--43.

\bibitem{yerocartpartres} I. G. Yero and J. A. Rodr\'{\i}guez-Vel\'{a}zquez, A note on the partition dimension of Cartesian product graphs, {\it Applied Mathematics and Computation} {\bf 217} (7) (2010) 3571--3574.

\bibitem{CMWA}    I. G. Yero, D. Kuziak, and J. A. Rodr\'{\i}guez-Vel\'{a}zquez, On the metric dimension of corona product graphs, \emph{Computers} \& \emph{Mathematics with Applications } \textbf{61} (9) (2011) 2793--2798.

\bibitem{vertex-transitive} H. Zhang, Independent sets in direct products of vertex-transitive graphs. Manuscript. {\em http://arxiv.org/pdf/1007.0797v1.pdf}

\end{thebibliography}
\end{document}